\documentclass[11pt,a4paper]{article}
\pdfoutput=1 % for arxiv
\usepackage{epsf,epsfig,amsfonts,amsgen,amsmath,amstext,amsbsy,amsopn,amsthm,cases,listings,color
}
\usepackage{ebezier,eepic}
\usepackage{color}
\usepackage{multirow}
\usepackage{epstopdf}
\usepackage{graphicx}
\usepackage{pgf,tikz}
\usepackage{mathrsfs}
\usepackage[marginal]{footmisc}
\usepackage{enumitem}
\usepackage[titletoc]{appendix}
\usepackage{booktabs}
\usepackage{url}
\usepackage{relsize}
\usepackage{mathtools}
\usepackage{comment}
\usepackage{pgfplots}
\usepackage{array}
\usepackage{authblk}
\usepackage{amssymb}
\usepackage{float}
\usepackage{wasysym}
\usetikzlibrary{positioning}
\usepackage{empheq}
\usepackage{wasysym} 
\usepackage{dsfont}

\usepackage{tikz}
\usepackage{longtable}
\usepackage{subfigure}
\pgfplotsset{compat=1.17}
\usepackage{mathrsfs}
\usepackage{wasysym}
\usetikzlibrary{arrows}
\usepackage{aligned-overset}
\usepackage{bm}%for vector
\usepackage{bbm}%for vector
\usepackage[T1]{fontenc}%for polish hook
\usepackage[backref=page]{hyperref}
\allowdisplaybreaks[1]

\definecolor{uuuuuu}{rgb}{0.27,0.27,0.27}
\definecolor{sqsqsq}{rgb}{0.1255,0.1255,0.1255}

\newtheorem{definition}{Definition} [section]
\newtheorem{theorem}[definition]{Theorem}
\newtheorem{lemma}[definition]{Lemma}

\newtheorem{claim}[definition]{Claim}

\newtheorem{fact}[definition]{Fact}

\newenvironment{pf}{\noindent{\bf Proof}}

\newenvironment{poc}{\begin{proof}[Proof of the claim]}{\end{proof}}
\newcommand{\dist}{\operatorname{dist}}

\newcommand{\hide}[1]{}

\begin{document}
\title{\bf\Large On cliques in hypergraphs}
\date{\today}
\author{Jun Gao\thanks{Research supported by ERC Advanced Grant 101020255. Email: \texttt{jun.gao@warwick.ac.uk}}}
\affil{Mathematics Institute and DIMAP,
            University of Warwick,
            Coventry, UK}
\maketitle
%%%%%%%%%%%%%%%%%%%%%%%%%%%%%%%%%%%%%%%%%%%%%
\begin{abstract}
We prove that for any $k \ge 3$, every $k$-uniform hypergraph on $n$ vertices contains at most $n - \omega(1)$ different sizes of cliques (maximal complete subgraphs). In particular, the 3-uniform case answers a question of Erd\H{o}s.
\end{abstract}

\section{Introduction}
Given an integer $k\ge 2$ and a $k$-uniform hypergraph $\mathcal{H}$, we say a vertex set $S$ of $\mathcal{H}$ forms a complete graph if every $k$-vertex subset of $S$ is an edge in $\mathcal{H}$. A complete subgraph of $\mathcal{H}$ is called a clique if it is not contained in any other complete subgraph of $\mathcal{H}$.

In this paper we study how many distinct clique sizes that can occur in a $k$-uniform hypergraph on $n$ vertices. We denote by 
$g(n,k)$ the largest possible number of distinct clique sizes in a $k$-uniform hypergraph of $n$ vertices.
For graphs (i.e., when $k=2$), such problems were first studied by Moon and Moser \cite{MoonMoser65} in 1965. They proved (throughout this paper $\log$ will denote logarithm to the base 2) that
\[
 n-\log n-2\log\log n < g(n,2) \le n-\lfloor\log n \rfloor.
\]
Subsequently, Erd\H{o}s \cite{Erdos66} improved the lower bound to 
\[
  g(n,2)> n-\log n-\log_* n -O(1),
\]
Where $\log_* n$ is the number of iterated logarithms such that $\log \log \dots \log n <1$.
Erd\H{o}s conjectured that this represented the correct order of magnitude.  Spencer \cite{Spencer71} later disproved this by proving that
\[
g(n,2)>n-\log n -O(1) .
\]

For hypergraphs, Erd\H{o}s (see \cite{Problem83}, or Erd\H{o}s Problem $\#$775 \cite{BloomErdos775}) showed that there is a 3-uniform hypergraph on $n$ vertices which contains $n- \log_* n$ cliques of distinct size and asked whether it can be shown that there is no 3-uniform hypergraph with $n-C$ cliques of distinct sizes, for some constant $C$ and large $n$. In this paper, we prove that such a hypergraph does not exist for any 
uniformity $k$, thereby answering this question.
\begin{theorem}\label{thm: main}
    For any integers $k \ge 3$ and $C\ge 0$, there exists a constant $N=N(k,C)$ such that for all $n\ge N$, every $n$-vertex $k$-uniform hypergraph contains no more than $n-C$ different sizes of cliques.
\end{theorem}
\section{Proof of Theorem~\ref{thm: main}}
For any hypergraph $\mathcal{H}$, we denote by $V(\mathcal{H})$ and $E(\mathcal{H})$ the vertex set and edge set, respectively. For a vertex $v$ in $\mathcal{H}$, denote by $N_\mathcal{H}(v)$ and $\deg_\mathcal{H}(v)$ the neighborhood and the degree of $v$ in $\mathcal{H}$, respectively.
Given a tree $T$ rooted at vertex $u$, let $x,y$ be two vertices in $T$. We use $\dist(x,y)$ to denote the length of the path between $x$ and $y$ in $T$, and define $\dist(x,x)=0$. 
Let $T(x)$ denote the subtree of $T$ rooted at $x$, i.e., $T(x)$ is the subtree of $T$ induced by $x$ and all vertices $v$ such that the unique path from $v$ to $u$ passes through $x$.
For any positive integer $n$, let $[n]$ denote the set $\{1,2,\cdots,n\}$.

Before proving the theorem, we introduce a definition that will be used in the argument.
\begin{definition}[$(k,C)$-layered tree]\label{def: tree}
    We say a tree $T$ is a $(k,C)$-layered tree if we can order vertices of the tree with $v_0,v_1,v_2,\cdots,v_t$ such that the following holds:
    \begin{enumerate}[label=(\arabic*)]
    \item Each vertex $v_i$ is a neighbor of some $v_j$ with $j<i$.
    \item For any $i\in [t]$, $\dist(v_0,v_i) \le k$.
    \item For each vertex $v_i\in  V(T)$, the degree of $v_i$ in $T$ is at most $2^{C+i}$.
    \end{enumerate}
\end{definition}

The next lemma shows that any $(k,C)$-layered tree has only a bounded number of vertices in terms of $k$ and $C$.
\begin{lemma}\label{lem: key}
    There exists a constant $N_0=N_0(C,k)$ such that any $(k,C)$-layered tree has at most $N_0$ vertices.
\end{lemma}
\begin{proof}
    Let $T$ be a $(k,C)$-layered tree with the vertices $v_0,v_1,v_2,\cdots,v_t$ satisfying the condition in Definition \ref{def: tree}.
    The proof proceeds by induction on $k$.
    
    When $k=1$, for any $i\in[t]$, $v_i$ is the neighbor of $v_0$ as $\dist(v_0,v_i) \le 1$. 
    It follows from $\deg_T(v_0) \le 2^C$ that
    $|V(T)| = t+1 \le 2^C+1$, so Lemma~\ref{lem: key} holds for $k=1$.

    Assume Lemma~\ref{lem: key} holds for $k-1$. Let $v_{x_1},v_{x_2}, \dots, v_{x_m}$ be the neighbors of $v_0$ in $T$ such that $x_1 < x_2<\dots < x_m$.
    Set $x_{m+1} = t+1$.
    For each $i\in[m]$.
    Let $T_i$ be an subtree of  $T$ with vertex set $\{v_0,v_1,v_2,\dots, v_{x_{i+1}-2}, v_{x_{i+1}-1}\}$. Let $T'_i$ be the tree obtained from $T_i$ by merging vertices $v_{x_1},v_{x_2}, \dots, v_{x_i}$ into a single new vertex $w_{i,0}$, replacing the edge with one endpoint in $\{v_{x_1},v_{x_2}, \dots, v_{x_i}\}$ by an edge incident to $w_{i,0}$ and then deleting the vertex $v_0$. Then we have $|V(T'_i)| = x_{i+1}-i$.
    \begin{claim}\label{claim: 2.3}
        If $\deg_T(v_{x_j}) \le C_j$ for any $j\in [i]$, then $T'_i$ is a $(k-1,2^C + C + \sum^{i}_{j=1}C_j)$-layered tree.
    \end{claim}
    \begin{poc}
    Given an integer $i$, we set $w_{i,0} = w_0$.
    The remaining vertices in the tree $T'_i$ are ordered in such a way $w_1,w_2,\cdots,w_{x_{i+1}-i-1}$ that their relative order is the same as in the original sequence, i.e., if $w_{i_1} =v_{j_1}$ and $w_{i_2} = v_{j_{2}}$, then $i_1<i_2$ if and only if $j_1<j_2$.
    Next, we prove that $T'_i$ is a $(k-1,2^C + C + \sum^{i}_{j=1}C_j)$-layered tree by verifying the three conditions in the definition with ordering $w_0,w_1,w_2\cdots,w_{x_{i+1}-i-1}$ .
    The first two conditions are automatically satisfied, as we do not change the order of the vertices except for $w_0$ and the vertex $v_0$ is removed. 
    
    Note that
        \[
        \deg_{T'_i}(w_0) = \sum_{j=1}^{i} \left(\deg_{T_i}(v_{x_j}) -1 \right) \le \sum_{j=1}^{i} \left(\deg_{T}(v_{x_j}) -1 \right) < \sum^{i}_{j=1}C_j\le 2^{\sum^{i}_{j=1}C_j}, 
        \]
        thus the third condition holds for $w_0$.
        For each $w_j \ne w_0$, there exists $j' \le j+2^C$ such that $w_j =v_{j'}$, as the size of the vertex set we are contracting is $i$, which satisfies $i\le m \le 2^C$.
        Note that for each $w_j \in T'_i$ with $w_j\ne w_0$, the degree of $w_j$ in $T'_i$, is at most the degree of $w_j=v_{j'}$ in $T$.
        We derive that
        \[
        \deg_{T'_i}(w_j)\le \deg_{T}(v_{j'})\le 2^{j'+C} \le 2^{j+2^C+C} \le  2^{j+2^C + C + \sum^{i}_{j=1}C_j}.
        \]
        So $T'_i$ is a $(k-1,2^C + C + \sum^{i}_{j=1}C_j)$-layered tree.
    \end{poc}
    \begin{claim}\label{claim: 2.4}
        For each $j\in [m]$, there exists a constant $C_j=C_j(C,k)$ such that $\deg_T(v_{x_j}) \le C_j$.
    \end{claim}
    \begin{poc}
        Let $m'$ be the maximum integer such that there exists a constant $C_{m'}=C_{m'}(C,k)$ such that $\deg_T(v_{x_{m'}}) \le C_{m'}$. It follows from $v_1=v_{x_1}$ is a neighbor of $v_0$ that we can take $C_1 = 2^{C+1}$, which implies that such a $m'$ exists.

        Suppose to the contrary that $m' <m$. By Claim~\ref{claim: 2.3}, $T'_{m'}$ is a $(k-1,2^C + C + \sum^{m'}_{j=1}C_j)$-layered tree.
        By the inductive hypothesis, we know that there exists a constant \[N_{m'} = N_0\left(k-1,2^{2^C + C + \sum^{m'}_{i=1}C_i}\right) \text{ such that }
        |V(T'_{m'})| = x_{m'+1}-m' \le N_{m'}.\] 
        Take $C_{m'+1} = 2^{N_{m'} +2^C+C} $, then $C_{m'+1}$ is a constant in terms of $k$ and $C$ with
        \[\deg_{T}(v_{x_{m'+1}}) \le 2^{x_{m'+1}+C} \le 2^{N_{m'}+m'+C} \le C_{m'+1}, \]
        a contradiction.
    \end{poc}
    By Claim~\ref{claim: 2.3}, $T'_m$ is a $(k-1,2^C + C + \sum^{m}_{i=1}C_i)$ layered tree. 
    By the inductive hypothesis and Claim~\ref{claim: 2.4}, there exists a constant $N_m$ in terms of $k$ and $C$ such that $|V(T'_m)| = t+1-m \le N_m$, which implies that
    \[
    |V(T)| = t+1 \le N_m+m \le N_m+2^C.
    \]
    So Lemma~\ref{lem: key} holds for $k$.
\end{proof}

The following fact will be useful in our proof.
\begin{fact}\label{fact: clique}
    Let $\mathcal{H}$ be a $k$-uniform hypergraph and $S_1,S_2,\dots, S_{k+1}$ be $k+1$ vertex sets of $\mathcal{H}$. If for each $i\in [k+1]$, $(\bigcup_{j=1,j\ne i}^{k+1}S_j) $ forms a complete graph in $\mathcal{H}$, then $(\bigcup_{j=1}^{k+1}S_j)$ forms  a complete graph in $\mathcal{H}$.
\end{fact}
Now we are ready to prove Theorem~\ref{thm: main}.
\begin{proof}[Proof of Theorem~\ref{thm: main}]
Let $\mathcal{H}$ be a $k$-uniform hypergraph on $n$ vertices. Suppose that $\mathcal{H}$ contains at least $n-C$ different sizes of cliques. 
Let $X_0,X_1,X_2,\dots, X_t$ be the collection of cliques in $\mathcal{H}$ such that $|V(X_0)| \ge |V(X_1)| \ge \dots \ge |V(X_t)|$.
Since $\mathcal{H}$ contains at least $n-C$ different sizes of cliques, 
we derive that 
\[
t+1\ge n-C \quad \text{and} \quad|V(X_i)| \ge n-C-i \quad \text{for}\quad i =0,1,2,\dots ,t.
\]We will show that there exists a $(k-1, C)$-layered tree $T$ with $V(T)$=$\{x_0,x_1,\dots, x_t\}$. If such a tree exists, then by Lemma~\ref{lem: key}, there exists a constant $N_0 = N_0(k-1,C)$ such that 
\[n-C \le t+1=|V(T)|\le N_0,
\] 
which implies that $n \le N_0+C$. So if we take $N =N_0 +C+1$, then every $k$-uniform hypergraph on $n$ vertices with $n\ge N$ contains no more than $n-C$ different
sizes of cliques.
So we only need to show such a tree exists.

Let $A_0 = V(X_0), B_0 = V(\mathcal{H})\setminus A_0$ and $T_0$ be the tree with single vertex $x_0$.
We now construct a sequence trees $T_0 \subseteq T_1 \subseteq T_2 \subseteq \dots \subseteq T_t = T$ rooted at $x_0$ for each $i\in \{0,1,2,\dots,t\}$ as follows:
\begin{itemize}
    \item Assume that for some integer $i$ with $0\le i<t$, we have constructed a tree $T_i$ on the vertex set $\{x_0,x_1,\dots x_i\}$, together with two sequences of sets $A_0,A_1,\dots,A_i$ and  $B_0,B_1,\dots,B_i$. Set $T^* = T_i$, $u=x_0$. For ease of notation, for each $0\le j\le t$, we define $X_{x_j} =X_j$, $A_{x_j} =A_j$ and $B_{x_j} = B_j$, respectively. We construct the new tree $T_{i+1}$ by adding the vertex $x_{i+1}$ following the rule described below:
    \begin{enumerate}  
    \item  If for every vertex $v$ adjacent to $u$ in $T^*$, $V(X_v)\cap B_u \ne V(X_{i+1}) \cap B_u $ or if $u$ has no neighbor in $T^*$, then let $T_{i+1}$ be the tree obtained from $T_{i}$ by adding the edge $u x_{i+1}$, and set $A_{i+1} = A_u\cap V(X_{i+1})$ and $B_{i+1} = A_{u} \setminus A_{i+1}$.
    \item Otherwise there exists vertex $w$ adjacent to $u$ in $T^*$ such that $V(X_w)\cap B_u = V(X_{i+1}) \cap B_u $. Then we reset $T^* = T_i(w)$ be the subtree of $T_i$ rooted at $w$, and reset $u=w$. Then return to Step 1.
    \end{enumerate}
    \item Since after each iteration, the distance between $x_0$
 and $u$ increases by one, the process must terminate after finitely many steps. Consequently, we obtain tree $T_{i+1}$ and two sets $A_{i+1},B_{i+1}$.
\end{itemize}
Now we claim that $T$ is a $(k-1,C)$-layered tree with the vertex ordering $x_0,x_1,\dots,x_t$ satisfying the
condition in Definition~\ref{def: tree}.
By the construction of 
$T$, we know that Condition (1) holds automatically. 

Before proving Conditions (2) and (3), we first establish some properties of the paths in $T$.
Let $P= y_0y_1y_2\dots y_r$ be a path of length $r$ with $x_0 =y_0$ as an endpoint.
Let $S_i = V(X_{y_i})\cap B_{y_{i-1}} $ for $i\in [r]$ and $S_0 = \emptyset$.
By the construction of $T$, the following holds:
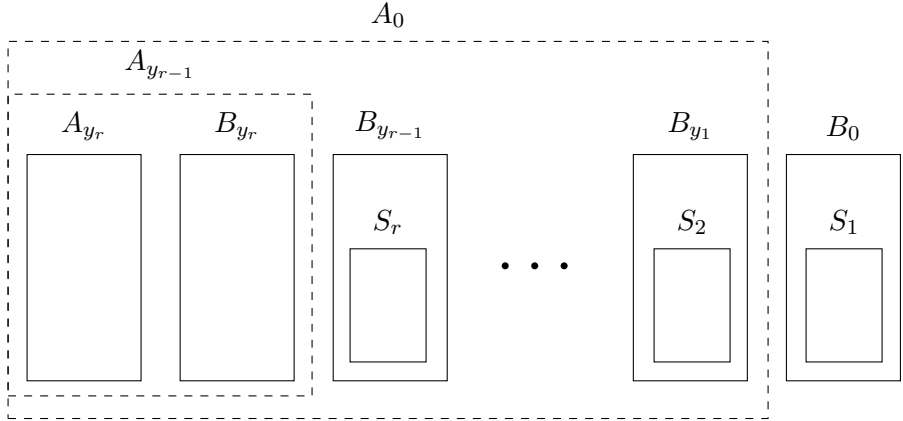
\begin{figure}[ht]
    \centering
\begin{tikzpicture}
    \node[draw, minimum width=1.5cm, minimum height=3cm] at (0,0) (r1) {};
    \node[above=0.5mm of r1] {$A_{y_r}$};
    \node[draw, minimum width=1.5cm, minimum height=3cm, right=0.5cm of r1] (r2) {};
    \node[above=0.5mm of r2] {$B_{y_r}$};
    \node[draw, minimum width=1.5cm, minimum height=3cm, right=0.5cm of r2] (r3) {};
    \node[above=0.5mm of r3] {$B_{y_{r-1}}$};

    \node[right=0.5cm of r3](doc) {\Huge $\cdots$};
    
    \node[draw, minimum width=1.5cm, minimum height=3cm, right=0.5cm of doc] (r4) {};
    \node[above=0.5mm of r4] {$B_{y_1}$};
    \node[draw, minimum width=1.5cm, minimum height=3cm, right=0.5cm of r4] (r5) {};
    \node[above=0.5mm of r5] {$B_0$};

    \node[draw, minimum width=1cm, minimum height=1.5cm] (r6) at (4,-0.5)  {};
    \node[draw, minimum width=1cm, minimum height=1.5cm] (r7) at (8,-0.5)  {};
    \node[draw, minimum width=1cm, minimum height=1.5cm] (r8) at (10,-0.5)  {};

    \node[draw,dashed,minimum width=4cm, minimum height=4cm] (r9) at (1,0.3)  {};
    \node[draw,dashed,minimum width=10cm, minimum height=5cm] (r10) at (4,0.5)  {};
    \node[above=0.5mm of r6] {$S_{r}$};
    \node[above=0.5mm of r7] {$S_2$};
    \node[above=0.5mm of r8] {$S_1$};
    \node[above=0.5mm of r9] {$A_{y_{r-1}}$};
    \node[above=0.5mm of r10] {$A_{0}$};
\end{tikzpicture}
\caption{Properties of the path in $T$.}
    \label{fig}
\end{figure}
\begin{itemize}
    \item For any $i \le r$, $A_{y_{i-1}}= A_{y_i} \sqcup B_{y_i},$ which implies that $V(\mathcal{H}) =A_0 \sqcup B_0 =\dots = A_{y_r} \sqcup B_{y_r} \sqcup B_{y_{r-1}} \sqcup \dots \sqcup B_0$, where we use $\sqcup$ to denote the disjoint union.
    \item For distinct $i$ and $j$, $S_i= V(X_{y_i})\cap B_{y_{i-1}}$ and $S_j = V(X_{y_j})\cap B_{y_{j-1}}$ are disjoint as $B_{y_{i-1}}$ and $B_{y_{j-1}}$ are disjoint.
    \item For each pair $i,j$ with $1\le i<j \le r$, since $y_j$ is not a neighbor of $y_{i-1}$ and $y_j\in T(y_i)$, by the process of adding the vertex $y_{j}$ to the tree, we derive that 
    \[V(X_{y_j}) \cap B_{y_{i-1}}=V(X_{y_i}) \cap B_{y_{i-1}} =S_i.\] 
    This means that for any $v,w\in  T(u)$ with $w\in T(v)$ and $v\ne u$, $X_v \cap B_u = X_w \cap B_u$.
    
    \item For each $j \ge 1$, by the definition of $A_{y_{j}}$, we have $V(X_{y_j})\cap A_{y_{j-1}} =A_{y_j}$. Then we derive that
    \[
    V(X_{y_j})= V(X_{y_j})\cap V(\mathcal{H}) =  V(X_{y_j}) \cap \left(A_{y_{j-1}}\sqcup \bigsqcup^{j-1}_{i=0} B_{y_{i}}\right) =  A_{y_j}\sqcup\bigsqcup^{j}_{i=1}  S_i  .
    \]
    \item For any $j\in[r]$, since $V(X_{y_i}) \not\subseteq V(X_{y_{i-1}})$ and
    \[V(X_{y_{j}}) = A_{y_j}\sqcup \bigsqcup^j_{i=1}S_i \subseteq A_{y_{j-1}}\cup S_j \cup \bigcup^{j-1}_{i=1}S_i = S_j \cup V(X_{y_{j-1}}),
    \] we have $S_j \ne \emptyset$, which implies that if $v\in  T(u)$ with $v\ne u$, then $X_v \cap B_u \ne \emptyset$.
\end{itemize}

To prove (2), we have the following claim.
\begin{claim}
    For any $i\in [t]$, $\dist(x_0,x_i) \le k-1$.
\end{claim}
\begin{poc}
    Suppose, for contradiction, that there exists a path of length $k$ with $x_0$ as an endpoint. 
    Without loss of generality, we assume that $y_0y_1\dots y_k$ is a path of length $k$ in $T$ with $x_0=y_0$.     
Let $U = A_{y_k} \cup S_k$, $W=A_{y_{k-1}}$, then we have
\begin{align*}
    &U\cup S_{k-1} \cup S_{k-2} \cdots \cup S_1 = V(X_{y_k}), \\
    &W \cup S_{k-1} \cup S_{k-2} \cdots \cup S_1 = V(X_{y_{k-1}}), 
\end{align*}
For each $j\in [k-1]$, it follows from $S_i \subseteq B_{y_{i-1}}$ and $A_{y_{i-1}} = A_{y_{i}} \cup B_{y_{i}}$ for any $i\in[k]$  that
\begin{align*}
    U\cup W \cup \bigcup_{i=1, i\ne j}^{k-1}S_i &= A_{y_{k-1}}\cup \bigcup^{k}_{i=j+1} S_{i}\cup \bigcup^{j-1}_{i=1} S_i \\
    &\subseteq  A_{y_{k-1}}\cup \bigcup^{k}_{i=j+1} B_{y_{i-1}}\cup \bigcup^{j-1}_{i=1} S_i\\
    &= A_{y_{j-1}}\cup \bigcup^{j-1}_{i=1} S_i=V(X_{y_{j-1}}).      
\end{align*}
By Fact~\ref{fact: clique}, $ V(X_{y_k})\cup V(X_{y_{k-1}}) = U\cup W \cup  S_{k-1} \cup S_{k-2} \cdots \cup S_1$ forms a complete graph, which contradicts the maximality of $X_{y_k}$ and $X_{y_{k-1}}$.
\end{poc}
Finally, we prove that Condition (3) holds. By the construction of  $T$, we know that for any vertex $x_i$, the new vertices added adjacent to $x_i$ correspond to distinct subsets of $B_i$. Moreover, each of these subsets is nonempty. Therefore, the degree of $x_0$ in T is at most $2^{|B_0|}-1 <  2^{|B_0|}$, and the degree of $x_i$ for $i\ne 0$ is at most $1+2^{|B_i|}-1 = 2^{|B_i|}$.
Since $B_i \cap V(X_i) = \emptyset$, we have $|B_i| \le n- |V(X_i)| \le C+i$, which implies that $\deg_T(x_i) \le 2^{C+i}$ for each $i\in \{0,1,2,\dots, t\}$.
Thus, $T$ is a $(k-1,C)$-layered tree. This completes the proof of Theorem~\ref{thm: main}.
\end{proof}
\section{Concluding remarks} \label{sec:remark}
If we let $f(n,k) = n-g(n,k)$, then in this paper, we prove that for every $k \ge 3$, $f(n,k) \to \infty$ as $n \to \infty$. A natural question arising from this result is: what growth rate can the function $f(n,k)$ achieve based on this proof? Erd\H{o}s showed that $f(n,3) \le \log_* n$. Using our approach, we can obtain
$f(n,3) \ge \log\!\bigl(\log_*(n+1)\bigr) - 1$ by applying the following lemma.
\begin{lemma}\label{lem: max tree}
    Let $T$ be a $(k,C)$-layered tree with a maximum number of vertices and vertex order $v_0, v_1, \dots, v_t$. Then $v_0, v_1, \dots, v_t$ corresponds to a depth-first search (DFS) of $T$.
\end{lemma}
\begin{proof}
Suppose, to the contrary, that $v_0, v_1, \dots, v_t$ is not a DFS of $T$. 
Let $j$ be the smallest integer that does not satisfy the DFS process. 
Since $T$ is $(k, C)$-layered, there exists $j' < j$ such that $v_{j'}v_j \in T$. 
Then there exists an index $i<j$ with $\dist(v_0, v_i) > \dist(v_0, v_{j'})$ 
and some $i' > j$ such that $v_i v_{i'} \in T$. 
Let $T'$ be the tree obtained from $T$ by adding a new vertex $w$ 
and a new edge $w v_i$, deleting the edge $v_i v_{i'}$, 
and adding the edge $v_{i'} v_j$. 
It is straightforward to verify that $T'$ is also a $(k, C)$-layered tree 
with vertex order 
$v_0, v_1, \dots, v_{j-1}, w, v_j, \dots, v_t,$
which leads to a contradiction.
\end{proof}
We define a sequence $a_0, a_1, \dots, a_{2^C}$ by setting $a_0 = 1$ and, for $i \ge 1$, 
\[
a_i = 2^{a_{i-1} + C} + a_{i-1,}\text{ which implies that $a_i+C+1 \le 2^{a_{i-1}+C+1}.$}
\]
By Lemma~\ref{lem: max tree} and defination of layerd tree, we know that the maximum number of vertices in a $(2,C)$-layered tree is $a_{2^C}$.
Let $\mathcal{H}$ be a $3$-uniform hypergraph on $n$ vertices that contains $n - C$ distinct clique sizes. Then $C \ge 2$.
Define $\log^{(1)} n = \log n$ and $\log^{(i+1)} n = \log\bigl(\log^{(i)} n\bigr)$ for $i \ge 1$.
It follows from the proof of Theorem~\ref{thm: main} that $n - C \le a_{2^C}$, which implies that
$
\log^{(2^C)} (n+1) \le C + 2.
$
If $C < \log\!\bigl(\log_*(n+1)\bigr) - 1$, then
$2^{C} \cdot 2 < \log_*(n+1),$
which  yields
$\log^{(2^C)}(n+1) > 2^{C} \ge C + 2,$
a contradiction. Therefore,
\[
C \ge \log\!\bigl(\log_*(n+1)\bigr) - 1.
\]

\section*{Acknowledgements}
The author thanks Shumin Sun for reading a draft of this paper and providing helpful comments, Stijn Cambie for a correction in the abstract, and Jie Ma and Haotian Yang for their comments.
\bibliographystyle{abbrv}
\bibliography{clique}
\end{document}